\documentclass[11pt]{article}

\usepackage{amsmath,amssymb,amsthm,mathtools}
\usepackage[T1]{fontenc}
\usepackage{lmodern}
\usepackage{microtype}
\usepackage{hyperref}
\usepackage{xurl}
\usepackage{enumitem}

\newtheorem{theorem}{Theorem}[section]
\newtheorem{lemma}[theorem]{Lemma}
\newtheorem{proposition}[theorem]{Proposition}
\newtheorem{corollary}[theorem]{Corollary}
\theoremstyle{definition}

\theoremstyle{remark}
\newtheorem{remark}[theorem]{Remark}

\newcommand{\SphD}{\operatorname{SphD}}
\newcommand{\artanh}{\operatorname{artanh}}

\title{The Probability That the Incenter of a\\Triangle Lies in a Random Diameter Disk}
\author{Stanley Rabinowitz}
\date{}

\begin{document}
\maketitle

\begin{abstract}
Let $P$ and $Q$ be independent points chosen uniformly from the interior of a nondegenerate triangle $ABC$, and let $I$ be its incenter.  We study the probability that the closed disk with diameter $PQ$ contains $I$.  In the language of multivariate statistics, this is the spherical depth of $I$ with respect to the uniform distribution on the triangle.

We first give an elementary planar form of the normalized cone-measure construction.  If $O$ is an interior point of a convex polygon, then the direction from $O$ to a uniformly distributed interior point has the same law as the direction from $O$ to a boundary point whose density on each side is proportional to the distance from $O$ to that side.  Consequently, this boundary point is uniform in arclength if and only if the polygon is tangential with incircle center $O$; for a triangle, this characterizes the incenter.

Using this transfer principle, we obtain the closed formula
\[
\SphD(I)
=\left(\frac r s\right)^2
\left[
\frac{8R}{r}-1
-\Gamma(\cos A)-\Gamma(\cos B)-\Gamma(\cos C)
\right],
\]
where $r,R,s$ are the inradius, circumradius, and semiperimeter, and
\[
\Gamma(t)=\frac1t-\frac{1-t^2}{t^2}\artanh t
\]
with continuous values $\Gamma(0)=0$ and $\Gamma(\pm1)=\pm1$.  Finally we prove the sharp inequality
\[
\SphD(I)\le \frac13+\frac{\log 3}{6},
\]
with equality if and only if $ABC$ is equilateral.
\end{abstract}



\section{Introduction}

Choose two points $P$ and $Q$ independently and uniformly at random
from the interior of a triangle $ABC$.  By ``uniformly'' we mean that
the probability of choosing a point in any subregion of the triangle is
proportional to the area of that subregion.  Construct the closed disk
having $PQ$ as a diameter, and let $I$ be the incenter of $ABC$.  We ask
the following question:
\[
\textit{What is the probability that the disk with diameter $PQ$
contains $I$?}
\]
By Thales' theorem,
\[
 I\in\operatorname{disk}(PQ)
 \quad\Longleftrightarrow\quad
 \angle PIQ\geq\frac{\pi}{2}
 \quad\Longleftrightarrow\quad
 (P-I)\cdot(Q-I)\leq0.
\]
Thus the problem may equivalently be viewed as asking for the
probability that two independent random points of the triangle subtend
an angle of at least $90^\circ$ at the incenter.

The problem has a useful geometric simplification.  Since the condition
above depends only on the directions of $P$ and $Q$ from $I$, it is
enough to understand how those directions are distributed.  We show
that choosing a point uniformly from the interior of a triangle induces
exactly the same distribution of directions from the incenter as
choosing a point uniformly by arclength from the boundary.  More
generally, this property characterizes the incenter among all interior
points of a triangle.  Using this observation, we obtain an explicit
formula for the probability above and prove that it is largest when the
triangle is equilateral.

In the terminology of multivariate statistics, the probability under
consideration is an instance of \emph{spherical depth}.  For a
probability distribution $F$ on $\mathbb R^d$, the spherical depth of a
point $q$ is the probability that $q$ belongs to the closed ball having
two independent $F$-distributed random points as endpoints of a
diameter.  This notion was introduced by Elmore, Hettmansperger, and
Xuan \cite{ElmoreHettmanspergerXuan}; see also the planar computational
work of Bremner and Shahsavarifar \cite{BremnerShahsavarifar}.  No
knowledge of statistical depth is needed here: throughout this paper,
$F$ is simply uniform area measure on a triangle and the distinguished
point $q$ is its incenter.

The transfer from uniform area measure to boundary measure is a special
case of the familiar cone-measure construction in convex geometry; see,
for example, \cite{Schneider,BoroczkyHenk}.  For polygons the relevant
density has a particularly elementary form, and uniform arclength
measure occurs precisely when the polygon is tangential and the
distinguished point is the center of its incircle.  In the case of a
triangle, this gives the characterization of the incenter mentioned
above.

Geometric-probability problems involving random points and a fixed
distinguished point have a substantial literature; for a related
example involving random triangles containing a fixed point, see
Iona\c{s}cu \cite{Ionascu}.  The equilateral specialization of the
diameter-disk problem was also obtained in a 2025 Mathematics Stack
Exchange discussion \cite{DanMSE}, where the value
\[
 \frac13+\frac{\log3}{6}
\]
was recorded.  We are not aware of a previous closed-form evaluation
for the incenter of an arbitrary triangle.

Our main results are as follows.  Section~\ref{sec:cone} gives the
planar cone-measure transfer theorem and the characterization of the
incenter.  Section~\ref{sec:formula} evaluates the probability that the
incenter lies in the random diameter disk.  Section~\ref{sec:max} proves
that the equilateral triangle uniquely maximizes this probability among
all triangles.  A technical one-variable inequality needed in the sharp
estimate is proved in Appendix~\ref{sec:appendix}.

\section{Directional cone measure and a characterization of the incenter}
\label{sec:cone}

Let $\mathcal P$ be a convex polygon of area $K$, and let $O$ be an interior point.  For a boundary point $X\ne O$, write
\[
\operatorname{dir}_O(X)=\frac{X-O}{|X-O|}.
\]
Let the sides of $\mathcal P$ be $S_1,\dots,S_n$, and let $d_i$ be the perpendicular distance from $O$ to the line containing $S_i$.

\begin{theorem}[Directional transfer]
\label{thm:transfer}
Let $P$ be uniformly distributed in $\mathcal P$.  Let $X$ be a random boundary point with probability density
\begin{equation}
\label{eq:boundarydensity}
 d\mu_O(X)=\frac{d_i}{2K}\,ds,
 \qquad X\in S_i,
\end{equation}
where $ds$ denotes arclength.  Then
\[
\operatorname{dir}_O(P)
\quad\text{and}\quad
\operatorname{dir}_O(X)
\]
have the same distribution.
\end{theorem}

\begin{proof}
Let $UV$ be a subsegment of a side $S_i$.  The set of points $P\in\mathcal P$ for which the ray $OP$ meets the boundary in $UV$ is exactly the triangle $OUV$, apart from sets of area zero.  Hence
\[
\Pr\{OP\text{ meets }UV\}
=\frac{[OUV]}K
=\frac{d_i\,|UV|}{2K}.
\]
This is precisely the mass assigned to $UV$ by \eqref{eq:boundarydensity}.  Since subsegments generate the Borel sets on the boundary, the two directional laws agree.
\end{proof}

The measure $\mu_O$ is the normalized cone measure on the boundary, expressed in the particularly simple form available for polygons.  Notice also that
\[
\sum_{i=1}^n d_i|S_i|=2K,
\]
so \eqref{eq:boundarydensity} is normalized.

\begin{corollary}[Tangential polygons]
\label{cor:tangential}
The measure $\mu_O$ is normalized arclength measure on $\partial\mathcal P$ if and only if all the distances $d_i$ are equal.  Equivalently, $\mathcal P$ is tangential and its incircle is centered at $O$.
\end{corollary}

\begin{proof}
If $L$ is the perimeter, normalized arclength has density $1/L$.  Thus \eqref{eq:boundarydensity} is normalized arclength if and only if
\[
\frac{d_1}{2K}=\cdots=\frac{d_n}{2K}=\frac1L,
\]
which is equivalent to $d_1=\cdots=d_n$.  In that case the circle centered at $O$ with this common radius is tangent to every supporting side of the polygon, hence is an incircle.  The converse is immediate.
\end{proof}

\begin{corollary}[Probabilistic characterization of the incenter]
\label{cor:incenter}
Let $O$ be an interior point of a triangle $ABC$.  The following are equivalent.
\begin{enumerate}[label=\textup{(\roman*)}]
\item $O$ is the incenter of $ABC$;
\item the direction from $O$ to a uniformly distributed interior point has the same distribution as the direction from $O$ to a boundary point chosen uniformly by arclength.
\end{enumerate}
\end{corollary}

\begin{proof}
A triangle is tangential, and its unique point equidistant from the three sidelines is its incenter.  The result follows from Corollary~\ref{cor:tangential}.
\end{proof}

The transfer extends immediately to several independent random points.

\begin{corollary}[Directional transfer principle]
\label{cor:multi}
Let $P_1,\dots,P_m$ be independent uniform points of $\mathcal P$, and let $X_1,\dots,X_m$ be independent boundary points with distribution $\mu_O$.  Every event depending only on the directions
\[
OP_1,\dots,OP_m
\]
has the same probability after $P_j$ is replaced by $X_j$ for every $j$.  If $\mathcal P$ is tangential with incircle center $O$, the $X_j$ may be chosen uniformly by arclength.
\end{corollary}

\section{The spherical depth of the incenter}
\label{sec:formula}

Let $ABC$ be a nondegenerate triangle with angles $A,B,C$, incenter $I$, inradius $r$, circumradius $R$, semiperimeter $s$, and area $K=rs$.  Let $P,Q$ be independent uniform interior points.  We write
\[
\mathcal D(ABC)=\Pr\{I\in\operatorname{disk}(PQ)\}.
\]
The boundary of the disk has probability zero, so closed or open containment makes no difference to the probability.

By Thales' theorem,
\begin{equation}
\label{eq:thales}
I\in\operatorname{disk}(PQ)
\quad\Longleftrightarrow\quad
(P-I)\cdot(Q-I)\le0.
\end{equation}
Thus the event depends only on the two directions from $I$.  Corollary~\ref{cor:multi} therefore allows us to replace $P,Q$ by two independent points chosen uniformly by arclength from the perimeter.

Because the desired probability is scale invariant, we normalize $r=1$.  Put
\begin{equation}
\label{eq:xyz}
x=\cot\frac A2,
\qquad
y=\cot\frac B2,
\qquad
z=\cot\frac C2.
\end{equation}
The standard half-angle identities give
\begin{equation}
\label{eq:pq}
p:=x+y+z=xyz=\frac sr,
\qquad
q:=xy+yz+zx=1+\frac{4R}{r}.
\end{equation}
Under the normalization $r=1$, the perimeter is $2p$.

We now compute the arclength measure of favorable ordered pairs of boundary points.

\subsection{Two points on the same side}

Consider the side $BC$.  Let its point of tangency with the incircle be the origin of a signed arclength coordinate $t$, positive toward $C$.  Then
\[
-y\le t\le z.
\]
With the incenter as the vector origin, a point on this side has the form
\[
\mathbf n+t\boldsymbol\tau,
\]
where $\mathbf n$ and $\boldsymbol\tau$ are perpendicular unit vectors.  Hence two points with parameters $t,u$ satisfy \eqref{eq:thales} exactly when
\[
1+tu\le0.
\]
Since $yz>1$, the favorable measure on $BC\times BC$ is
\begin{align}
M_{BC,BC}
&=2\int_{1/y}^{z}\left(y-\frac1t\right)dt \\
&=2\bigl(yz-1-\log(yz)\bigr).
\label{eq:sameside}
\end{align}
Cyclic summation gives
\begin{equation}
\label{eq:samesidesum}
M_{\mathrm{same}}
=2q-6-4\log p.
\end{equation}

\subsection{Two points on adjacent sides}

We next consider $BC\times CA$, whose sides meet at $C$.  Put $\alpha=C/2$, and write
\[
z=\cot\alpha.
\]
Choose unit normals $\mathbf n_1,\mathbf n_2$ from $I$ to the two sides and unit tangents $\boldsymbol\tau_1,\boldsymbol\tau_2$ pointing from the tangency points toward $C$.  The points may be written
\[
X=\mathbf n_1+t\boldsymbol\tau_1,
\qquad
Y=\mathbf n_2+u\boldsymbol\tau_2,
\]
with
\[
-y\le t\le z,
\qquad
-x\le u\le z.
\]
Put
\[
t=\tan\phi,
\qquad
u=\tan\psi.
\]
Then
\[
\phi\in\left[\frac B2-\frac\pi2,\frac\pi2-\frac C2\right],
\qquad
\psi\in\left[\frac A2-\frac\pi2,\frac\pi2-\frac C2\right].
\]
A direct dot-product calculation gives
\[
\frac{X\cdot Y}{|X||Y|}
=
\cos\bigl(\pi-C-\phi-\psi\bigr).
\]
The angle on the right ranges from $0$ to
\[
\frac{3\pi}{2}-\frac C2<\frac{3\pi}{2}.
\]
Consequently the condition $X\cdot Y\le0$ is equivalent, throughout the relevant range, to
\begin{equation}
\label{eq:angularcondition}
\phi+\psi\le\frac\pi2-C.
\end{equation}

We shall express the resulting integral using the odd function
\begin{equation}
\label{eq:Gamma}
\Gamma(t)
=
\frac1t-\frac{1-t^2}{t^2}\artanh t,
\qquad -1<t<1,
\end{equation}
with continuous extensions
\[
\Gamma(0)=0,
\qquad
\Gamma(1)=1,
\qquad
\Gamma(-1)=-1.
\]

Integrating \eqref{eq:angularcondition}, and using $dt=\sec^2\phi\,d\phi$ and $du=\sec^2\psi\,d\psi$, we split at $\phi=-C/2$.  For the first part the whole $CA$-interval is admissible; for the second part the upper endpoint is determined by \eqref{eq:angularcondition}.  Hence
\begin{align*}
M_{BC,CA}
={}&\left(y-\frac1z\right)(x+z)
+x\left(z+\frac1z\right)+J_C\\
={}&q-1+J_C,
\end{align*}
where
\[
J_C
=
\int_{-C/2}^{\pi/2-C/2}
\sec^2\phi\,
\tan\left(\frac\pi2-C-\phi\right)d\phi.
\]
If $k=\cot C$ and $t=\tan\phi$, then
\[
J_C
=\int_{-1/z}^{z}\frac{k-t}{1+kt}\,dt.
\]
There is no pole on this interval: indeed,
\[
1+\cot C\tan\phi
=\frac{\sin(C+\phi)}{\sin C\cos\phi},
\]
and $C+\phi\in[C/2,\pi/2+C/2]\subset(0,\pi)$.
Partial fractions now give
\begin{equation}
\label{eq:JC}
J_C=-2\sec C+2\sec^2C\log z.
\end{equation}
Since
\[
\artanh(\cos C)=\log\cot\frac C2=\log z
\]
and
\[
\Gamma(\cos C)=\sec C-\tan^2C\log z,
\]
equation \eqref{eq:JC} becomes simply
\begin{equation}
\label{eq:JCgamma}
J_C=2\log z-2\Gamma(\cos C).
\end{equation}
Therefore
\begin{equation}
\label{eq:adjacent}
M_{BC,CA}=q-1+2\log z-2\Gamma(\cos C).
\end{equation}

Both orderings of an adjacent pair of sides contribute equally.  Thus
\begin{equation}
\label{eq:adjacentsum}
M_{\mathrm{adj}}
=6q-6+4\log p
-4\sum_{\mathrm{cyc}}\Gamma(\cos A).
\end{equation}
Combining this with \eqref{eq:samesidesum}, the logarithmic terms cancel:
\[
M_{\mathrm{same}}+M_{\mathrm{adj}}
=8q-12-4\sum_{\mathrm{cyc}}\Gamma(\cos A).
\]
Dividing by the square $(2p)^2$ of the perimeter yields
\begin{equation}
\label{eq:Gammaform}
\mathcal D(ABC)
=
\frac{2q-3-
\Gamma(\cos A)-\Gamma(\cos B)-\Gamma(\cos C)}{p^2}.
\end{equation}

\begin{theorem}[Exact spherical depth]
\label{thm:formula}
For every nondegenerate triangle $ABC$,
\begin{equation}
\label{eq:mainformula}
\boxed{
\mathcal D(ABC)
=
\left(\frac r s\right)^2
\left[
\frac{8R}{r}-1
-\Gamma(\cos A)-\Gamma(\cos B)-\Gamma(\cos C)
\right].
}
\end{equation}
\end{theorem}

\begin{proof}
Equation \eqref{eq:Gammaform}, together with $p=s/r$ and $q=1+4R/r$ from \eqref{eq:pq}, gives \eqref{eq:mainformula}.
\end{proof}

\begin{remark}
An equivalent half-angle form is obtained from
\[
\cos A=\frac{x^2-1}{x^2+1}.
\]
Namely,
\[
\Gamma(\cos A)
=
G(x),
\qquad
G(x)=\frac{x^4-1-4x^2\log x}{(x^2-1)^2},
\]
with $G(1)=0$.  Therefore
\[
\mathcal D(ABC)
=
\frac{2q-3-G(x)-G(y)-G(z)}{p^2}.
\]
\end{remark}

For an equilateral triangle, $R=2r$, $s/r=3\sqrt3$, and
\[
\Gamma\left(\frac12\right)=2-\frac32\log3.
\]
Thus Theorem~\ref{thm:formula} gives
\begin{equation}
\label{eq:eqvalue}
\mathcal D_{\mathrm{eq}}
=
\frac13+\frac{\log3}{6}
=0.5164353814\ldots.
\end{equation}

\section{The equilateral triangle is the unique maximizer}
\label{sec:max}

We now prove that \eqref{eq:eqvalue} is the largest possible value of \eqref{eq:mainformula}.

Put
\begin{equation}
\label{eq:kappa}
\kappa=10-9\log3.
\end{equation}
This value is forced by the equality case that we seek below.  At an equilateral triangle,
\[
3\Gamma\left(\frac12\right)
=6-\frac92\log3
=1+\frac\kappa2,
\]
so any sharp inequality of the form \eqref{eq:Gammaineq} must use precisely this constant.

The elementary series
\[
\log3
=2\sum_{n=0}^{\infty}\frac{1}{(2n+1)2^{2n+1}}
\]
gives the convenient bounds
\begin{equation}
\label{eq:kappabounds}
0<\kappa<\frac18.
\end{equation}
Indeed, bounding every denominator in the tail after the first term by $3$ gives
$\log3<10/9$, while the first four terms give
$\log3>7379/6720>79/72$.

The following elementary fourth-derivative lemma will be used twice.  Its formulation is chosen so that no endpoint differentiability is required.

\begin{lemma}[A fourth-derivative Rolle lemma]
\label{lem:rolle4}
Let $\varphi$ be continuous on $[\alpha,\beta]$ and four times differentiable on $(\alpha,\beta)$, and suppose
\[
\varphi^{(4)}(x)<0
\qquad(\alpha<x<\beta).
\]
Put $\mu=(\alpha+\beta)/2$, and assume
\[
\varphi(\alpha)=\varphi(\beta)=0,
\qquad
\varphi'(\mu)=0.
\]
If $d=\varphi(\mu)\ge0$, then
\begin{equation}
\label{eq:rollebound}
\varphi(t)>
\frac{4d}{(\beta-\alpha)^2}(t-\alpha)(\beta-t)
\qquad
(t\in(\alpha,\beta),\ t\ne\mu).
\end{equation}
In particular, $\varphi(t)>0$ for $t\ne\mu$, while $\varphi(\mu)=d\ge0$.
\end{lemma}

\begin{proof}
Let
\[
P(x)=\frac{4d}{(\beta-\alpha)^2}(x-\alpha)(\beta-x)
\]
and
\[
W(x)=(x-\alpha)(x-\mu)^2(x-\beta).
\]
Fix $t\in(\alpha,\beta)\setminus\{\mu\}$, and choose $\lambda$ so that
\[
G(x)=\varphi(x)-P(x)-\lambda W(x)
\]
satisfies $G(t)=0$.  By construction,
\[
G(\alpha)=G(\mu)=G(t)=G(\beta)=0,
\qquad
G'(\mu)=0.
\]
Rolle's theorem applied between the four distinct zeros of $G$ supplies three zeros of $G'$; together with the additional zero at $\mu$, these are four distinct zeros of $G'$ in $(\alpha,\beta)$.  Three further applications of Rolle's theorem therefore give a point $\xi\in(\alpha,\beta)$ for which $G^{(4)}(\xi)=0$.

Since $P^{(4)}=0$ and $W^{(4)}=24$,
\[
\lambda=\frac{\varphi^{(4)}(\xi)}{24}<0.
\]
Also $W(t)<0$ for $\alpha<t<\beta$, $t\ne\mu$.  Hence
\[
\varphi(t)-P(t)=\lambda W(t)>0,
\]
which is \eqref{eq:rollebound}.
\end{proof}

The central estimate is the following triangle inequality for $\Gamma$.

\begin{proposition}
\label{prop:Gammaineq}
For every nondegenerate triangle $ABC$,
\begin{equation}
\label{eq:Gammaineq}
\Gamma(\cos A)+\Gamma(\cos B)+\Gamma(\cos C)
\ge
1+\kappa\frac rR.
\end{equation}
Equality holds if and only if $ABC$ is equilateral.
\end{proposition}

We first convert this to a one-variable superadditivity statement.  Define
\begin{equation}
\label{eq:fdef}
f(\theta)
=
1-\kappa-\Gamma(\cos\theta)+\kappa\cos\theta,
\qquad 0\le\theta\le\pi.
\end{equation}
Because $\Gamma$ is odd,
\begin{equation}
\label{eq:fsym}
f(\pi-\theta)=2(1-\kappa)-f(\theta).
\end{equation}
In particular,
\begin{equation}
\label{eq:fzero}
f(0)=0.
\end{equation}
Also the standard identity
\begin{equation}
\label{eq:cosrho}
\cos A+\cos B+\cos C=1+\frac rR
\end{equation}
shows that \eqref{eq:Gammaineq} is equivalent to
\begin{equation}
\label{eq:super}
f(A+B)\ge f(A)+f(B).
\end{equation}
Indeed, since $C=\pi-(A+B)$, subtracting the right side of \eqref{eq:Gammaineq} from the left side gives exactly
\[
f(A+B)-f(A)-f(B).
\]

The midpoint case of \eqref{eq:super} is the technical ingredient proved in Appendix~\ref{sec:appendix}.

\begin{lemma}[Midpoint inequality]
\label{lem:midpoint}
For $0<t<\pi/2$,
\begin{equation}
\label{eq:midpoint}
f(2t)\ge2f(t),
\end{equation}
with equality if and only if $t=\pi/3$.
\end{lemma}

\begin{proof}
Put $u=\cos t$.  A direct use of \eqref{eq:fdef} and the oddness of $\Gamma$ gives
\[
f(2t)-2f(t)
=
\Gamma(1-2u^2)+2\Gamma(u)-1-2\kappa u(1-u).
\]
The right side is the function $H(u)$ of Lemma~\ref{lem:H} in Appendix~\ref{sec:appendix}.  That lemma states that it is nonnegative for $0<u<1$, with equality only at $u=1/2$, i.e. only at $t=\pi/3$.
\end{proof}

We next show that the midpoint case implies full superadditivity for this particular function $f$.

\begin{lemma}
\label{lem:superadd}
For $a,b>0$ with $a+b<\pi$,
\[
f(a+b)\ge f(a)+f(b).
\]
Equality holds if and only if $a=b=\pi/3$.
\end{lemma}

\begin{proof}
We begin with the power series
\begin{equation}
\label{eq:Gammaseries}
\Gamma(t)
=2\sum_{n=0}^{\infty}
\frac{t^{2n+1}}{(2n+1)(2n+3)},
\qquad |t|<1.
\end{equation}
Termwise differentiation in \eqref{eq:fdef} gives, for $0<\theta<\pi/2$,
\begin{equation}
\label{eq:f4series}
f^{(4)}(\theta)
=c_0\cos\theta
+\sum_{n=1}^{\infty}c_n\cos^{2n+1}\theta,
\end{equation}
where
\begin{align}
\label{eq:coeffs}
c_0&=\frac{220}{21}-9\log3=\kappa+\frac{10}{21}>0,\\
c_n&=\frac{2(2n+1)(4n^2+44n+25)}{(2n+3)(2n+5)(2n+7)}>0,
\qquad n\ge1.
\end{align}
Thus $f^{(4)}$ is positive and strictly decreasing on $(0,\pi/2)$.  From \eqref{eq:fsym},
\begin{equation}
\label{eq:f4sym}
f^{(4)}(\pi-\theta)=-f^{(4)}(\theta),
\end{equation}
and in particular $f^{(4)}(\pi/2)=0$.

Fix $s\in(0,\pi)$ and set
\[
D_s(t)=f(s)-f(t)-f(s-t),
\qquad 0\le t\le s.
\]
By \eqref{eq:fzero},
\[
D_s(0)=D_s(s)=0,
\]
and symmetry gives $D_s'(s/2)=0$.  Moreover
\begin{equation}
\label{eq:D4}
D_s^{(4)}(t)
=-f^{(4)}(t)-f^{(4)}(s-t)<0
\qquad(0<t<s).
\end{equation}
Indeed, if $t,s-t<\pi/2$, both terms in the sum are positive.  If, say, $s-t>\pi/2$, put
\[
v=\pi-(s-t)=\pi-s+t.
\]
Then $0<t<v<\pi/2$, so the strict decrease of $f^{(4)}$ and \eqref{eq:f4sym} give
\[
f^{(4)}(t)+f^{(4)}(s-t)
=f^{(4)}(t)-f^{(4)}(v)>0.
\]
If $s-t=\pi/2$, then $t<\pi/2$ and the same conclusion follows from $f^{(4)}(\pi/2)=0<f^{(4)}(t)$.  The case $t\ge\pi/2$ is symmetric.

At the midpoint,
\[
d:=D_s\left(\frac s2\right)
=f(s)-2f\left(\frac s2\right)\ge0
\]
by Lemma~\ref{lem:midpoint}.  Lemma~\ref{lem:rolle4}, applied on $[0,s]$ to $D_s$, now gives
\[
D_s(t)>\frac{4d}{s^2}t(s-t)\ge0
\]
for $t\ne s/2$, while $D_s(s/2)=d\ge0$.  Thus $D_s(t)\ge0$ throughout $[0,s]$, proving superadditivity.

If equality holds, Lemma~\ref{lem:rolle4} forces $t=s/2$; Lemma~\ref{lem:midpoint} then forces $s/2=\pi/3$.  Hence $a=b=\pi/3$.
\end{proof}

\begin{proof}[Proof of Proposition~\ref{prop:Gammaineq}]
Apply Lemma~\ref{lem:superadd} with $a=A$ and $b=B$.  As observed above, the resulting inequality is precisely \eqref{eq:Gammaineq}.  Equality in Lemma~\ref{lem:superadd} gives $A=B=\pi/3$, hence also $C=\pi/3$.
\end{proof}

\begin{remark}
Proposition~\ref{prop:Gammaineq} is also asymptotically sharp at the degenerate configuration $(A,B,C)\to(\pi/2,\pi/2,0)$.  Indeed,
\[
f(\pi/2)=1-\kappa,
\qquad
f(0)=0,
\]
so the equivalent superadditivity inequality approaches equality there.  The final maximum theorem remains strict away from the equilateral triangle because the subsequent Schur estimate is not sharp in this degeneration.
\end{remark}

We can now complete the extremal argument.

\begin{theorem}[Sharp maximum]
\label{thm:max}
For every nondegenerate triangle $ABC$,
\begin{equation}
\label{eq:max}
\boxed{
\mathcal D(ABC)
\le
\frac13+\frac{\log3}{6}.
}
\end{equation}
Equality holds if and only if $ABC$ is equilateral.
\end{theorem}

\begin{proof}
Retain the notation $p,q$ from \eqref{eq:pq}, and put
\[
T=\frac Rr\ge2.
\]
Theorem~\ref{thm:formula} and Proposition~\ref{prop:Gammaineq} give
\begin{equation}
\label{eq:firstbound}
\mathcal D(ABC)
\le
\frac{8T-2-\kappa/T}{p^2}.
\end{equation}

Schur's inequality applied to the positive numbers $x,y,z$ gives
\[
(x+y+z)^3+9xyz
\ge
4(x+y+z)(xy+yz+zx).
\]
Since $x+y+z=xyz=p$ and $q=1+4T$, this becomes
\begin{equation}
\label{eq:Gerber}
p^2\ge16T-5.
\end{equation}
Therefore
\begin{equation}
\label{eq:hT}
\mathcal D(ABC)
\le
h(T):=
\frac{8T-2-\kappa/T}{16T-5}.
\end{equation}
By \eqref{eq:kappabounds},
\[
h'(T)
=
\frac{-8+32\kappa/T-5\kappa/T^2}{(16T-5)^2}<0
\qquad(T\ge2).
\]
Hence $h(T)\le h(2)$, and
\[
h(2)
=
\frac{14-\kappa/2}{27}
=
\frac13+\frac{\log3}{6}.
\]
This proves \eqref{eq:max}.

If equality holds, Proposition~\ref{prop:Gammaineq} already forces the triangle to be equilateral.  Conversely, the equilateral triangle attains the value in \eqref{eq:eqvalue}.
\end{proof}

\begin{remark}
The inequality \eqref{eq:Gerber} is the classical relation
\[
s^2\ge16Rr-5r^2,
\]
here obtained immediately from Schur's inequality after the half-angle substitution \eqref{eq:xyz}.
\end{remark}

\section{Concluding remarks}

The directional-transfer theorem separates the problem into two parts.  The first is affine-measure-theoretic: uniform area measure induces normalized cone measure on the boundary.  The second is metric: for the incenter of a triangle, the cone measure is exactly normalized arclength, and the disk condition is the right-angle condition \eqref{eq:thales}.  This combination is what makes the spherical depth explicitly computable.

The characterization in Corollary~\ref{cor:incenter} also suggests other planar questions.  For example, if $O$ is an arbitrary interior point of a triangle, Theorem~\ref{thm:transfer} still reduces its spherical depth to a boundary integral, but the three sides carry different constant densities proportional to their distances from $O$.  Determining the point of maximum spherical depth for a fixed scalene triangle appears to be a separate problem.

\section*{Acknowledgment}
The author gratefully acknowledges substantial assistance from ChatGPT (OpenAI) in the exploration of the problem, the development and checking of formulas and proofs, and the preparation and revision of the manuscript.  Responsibility for the final arguments and presentation rests with the author.

\appendix
\section{The one-variable inequality}
\label{sec:appendix}

This appendix proves the inequality used in Lemma~\ref{lem:midpoint}.  We first record a rational lower bound for $\artanh$.

\begin{lemma}
\label{lem:atanhbound}
For $0<t<1$,
\begin{equation}
\label{eq:atanhbound}
\artanh t
>
\frac{t(945-735t^2+64t^4)}{15(63-70t^2+15t^4)}.
\end{equation}
\end{lemma}

\begin{proof}
Put $s=t^2$ and
\[
F(s)=\frac{\artanh\sqrt{s}}{\sqrt{s}}
=\sum_{n=0}^{\infty}\frac{s^n}{2n+1}.
\]
Let
\[
P(s)=1-\frac79s+\frac{64}{945}s^2,
\qquad
Q(s)=1-\frac{10}{9}s+\frac5{21}s^2.
\]
For $0\le s\le1$, $Q(s)>0$, and direct coefficient comparison gives
\begin{equation}
\label{eq:PQseries}
Q(s)F(s)-P(s)
=
\sum_{n=5}^{\infty}
\frac{32(n-4)(n-3)}{63(2n-3)(2n-1)(2n+1)}s^n>0.
\end{equation}
Thus $F(s)>P(s)/Q(s)$, which is exactly \eqref{eq:atanhbound}.
\end{proof}

We now prove the fourth-derivative sign needed below.

\begin{lemma}
\label{lem:H4}
Let
\begin{equation}
\label{eq:Hdef}
H(u)
=
\Gamma(1-2u^2)+2\Gamma(u)-1-2\kappa u(1-u),
\qquad 0<u<1,
\end{equation}
where $\kappa$ is given by \eqref{eq:kappa}.  Then
\begin{equation}
\label{eq:H4neg}
H^{(4)}(u)<0
\qquad(0<u<1).
\end{equation}
\end{lemma}

\begin{proof}
The quadratic term involving $\kappa$ disappears after four differentiations.  Put
\[
a=1-2u^2.
\]
For $u\ne1/\sqrt2$, direct differentiation gives
\begin{align}
H^{(4)}(u)={}&
\frac{8N_{16}(u)}
{u^5(u-1)^3(u+1)^3(2u^2-1)^5}
-\frac{240}{u^6}\artanh u \notag\\
&-\frac{96(140u^4+84u^2+3)}{(2u^2-1)^6}\artanh a,
\label{eq:H4exact}
\end{align}
where
\begin{align*}
N_{16}(u)={}&384u^{16}-480u^{15}-3072u^{14}+960u^{13}+8800u^{12}
+564u^{11}\\
&-13120u^{10}-1884u^9+11560u^8+743u^7-6272u^6+106u^5\\
&+2066u^4-u^3-380u^2+30.
\end{align*}
The apparent singularity at $u=1/\sqrt2$ is removable, since $H$ is smooth there.  Formula~\eqref{eq:H4exact} makes clear, in particular, that the coefficients of $\artanh u$ and $\artanh a$ are both negative.

First suppose $0<u<1/\sqrt2$, so $0<a<1$.  Apply Lemma~\ref{lem:atanhbound} to both inverse-hyperbolic-tangent terms.  Because both coefficients are negative, this gives an upper bound.
Define
\[
D_1(u)=(15u^4-70u^2+63)
(30u^8-60u^6+10u^4+20u^2+1).
\]
Then the resulting upper bound is
\begin{equation}
\label{eq:H4case1}
H^{(4)}(u)
\le
\frac{8P_{15}(u)}{5u^2(u^2-1)^3D_1(u)},
\end{equation}
where
\begin{align*}
P_{15}(u)={}&
7800u^{15}-150u^{14}-12000u^{13}-12515u^{12}
-20800u^{11}+78035u^{10}\\
&+38800u^9-138383u^8-2740u^7+88504u^6
-10680u^5-17089u^4\\
&-540u^3+1603u^2+315.
\end{align*}
This polynomial is positive for $0<u<1$.  Indeed, with $u=y/(1+y)$,
\begin{align*}
&(1+y)^{15}P_{15}\left(\frac{y}{1+y}\right)\\
={}&315+4725y+34678y^2+163624y^3+531440y^4
+1170104y^5\\
&+1645729y^6+1316297y^7+454964y^8+118180y^9
+397704y^{10}\\
&+494324y^{11}+279824y^{12}+93360y^{13}+8320y^{14}+160y^{15},
\end{align*}
whose coefficients are all positive.  The two factors defining $D_1(u)$ are positive on the present interval; in fact
\[
15u^4-70u^2+63>0
\]
and
\[
30u^8-60u^6+10u^4+20u^2+1
=\frac18\bigl(15a^4-70a^2+63\bigr)>0.
\]
Since $(u-1)^3(u+1)^3<0$, the right side of \eqref{eq:H4case1} is negative.

Now suppose $1/\sqrt2<u<1$ and put
\[
b=2u^2-1=-a\in(0,1).
\]
Then $\artanh a=-\artanh b$.  We use the elementary bounds
\begin{equation}
\label{eq:atanhlowertrunc}
\artanh t>
 t+\frac{t^3}{3}+\frac{t^5}{5}+\frac{t^7}{7}
\end{equation}
and
\begin{equation}
\label{eq:atanhuppertrunc}
\artanh t<
 t+\frac{t^3}{3}+\frac{t^5}{5}
 +\frac{t^7}{7(1-t^2)},
\qquad 0<t<1.
\end{equation}
The first follows by truncating the positive series for $\artanh t$; the second follows by bounding its remaining denominators by $7$.

Apply \eqref{eq:atanhlowertrunc} to $\artanh u$ and \eqref{eq:atanhuppertrunc} to $\artanh b$.  The result simplifies to
\begin{equation}
\label{eq:H4case2}
H^{(4)}(u)
\le
\frac{16P_{10}(u)}{35u^2(u-1)^3(u+1)^3},
\end{equation}
where
\begin{align*}
P_{10}(u)={}&
840u^{10}-75u^9-1596u^8+120u^7+928u^6
-85u^5\\
&+102u^4-30u^3-174u^2+40.
\end{align*}
This polynomial is positive for $u>3/5$.  Put
\[
u=\frac{3/5+y}{1+y},
\qquad y>0.
\]
Then
\begin{align*}
&5^9(1+y)^{10}P_{10}\left(\frac{3/5+y}{1+y}\right)\\
={}&3246602+16762640y+93221670y^2+536913600y^3
+1813552500y^4\\
&+3534460000y^5+4220537500y^6+3261000000y^7
+1738281250y^8\\
&+656250000y^9+136718750y^{10},
\end{align*}
again with all coefficients positive.  Since $1/\sqrt2>3/5$ and $(u-1)^3(u+1)^3<0$, the right side of \eqref{eq:H4case2} is negative.

The point $u=1/\sqrt2$ follows by continuity (or by taking the limit in either bound).  This proves \eqref{eq:H4neg}.
\end{proof}

\begin{lemma}
\label{lem:H}
For $0<u<1$,
\begin{equation}
\label{eq:Hpositive}
H(u)\ge0,
\end{equation}
where $H$ is defined by \eqref{eq:Hdef}.  Equality holds if and only if $u=1/2$.
\end{lemma}

\begin{proof}
Using the continuous endpoint values of $\Gamma$, one has
\[
H(0)=H(1)=0.
\]
Also
\[
\Gamma\left(\frac12\right)
=2-\frac32\log3,
\]
so the choice \eqref{eq:kappa} gives
\[
H\left(\frac12\right)=0.
\]
Finally,
\[
\Gamma'(t)=\frac{2(\artanh t-t)}{t^3},
\]
and the two $\Gamma'$ contributions in $H'$ cancel at $u=1/2$, while the derivative of $u(1-u)$ vanishes there.  Hence
\[
H'\left(\frac12\right)=0.
\]

Lemma~\ref{lem:H4} gives $H^{(4)}(u)<0$ for $0<u<1$.  Apply Lemma~\ref{lem:rolle4} to $H$ on $[0,1]$, with midpoint $1/2$ and $d=H(1/2)=0$.  It follows that
\[
H(u)>0
\qquad
(0<u<1,\ u\ne1/2),
\]
while $H(1/2)=0$.  This proves the lemma.
\end{proof}

\end{document}